\documentclass[12pt]{amsart}
\usepackage{amssymb,mathtools,enumitem,microtype}
\setlist[enumerate,1]{labelindent=\parindent,labelsep=0.5\parindent,leftmargin=*,align=left}
\numberwithin{equation}{section}
\newtheorem{thm}{Theorem}[section]
\newtheorem{prop}[thm]{Proposition}
\newtheorem{lem}[thm]{Lemma}

\theoremstyle{definition}
\newtheorem{defn}[thm]{Definition}
\theoremstyle{remark}
\newtheorem{remark}[thm]{Remark}

\newcommand{\Ef}{\mathcal E_f(0)}
\newcommand{\EFq}{\mathcal E_F(q)}
\newcommand{\Tau}{\mathcal T}
\renewcommand{\restriction}{\mathord{\upharpoonright}}

\begin{document}

\title[Absolute Winning for MP Maps]{Absolute Winning Exceptional Sets for Intermittent Interval Maps}
\author{Jason Duvall}
\email{jasonrduvall@gmail.com}

\begin{abstract}
	We prove that for a Manneville--Pomeau type interval map, the set of points whose orbit closures miss a prescribed countable set is absolute winning in the sense of McMullen. The proof has three parts. First we directly prove that the exceptional set for the distinguished endpoint of the induced first-return map is absolute winning. Then we use the finite-branch winning theorem of Hu--Li--Yu, together with the one-dimensional implication from $1/2$-strong winning to absolute winning, to obtain absolute winning for all countable induced targets. Finally, a quasisymmetric pullback argument transfers these induced results back to the original map.
\end{abstract}

\maketitle

\section{Introduction}

Let $T\colon X\to X$ be a dynamical system on a compact metric space and let $P \subset X$. We write
\begin{equation*}
	\mathcal{E}_T(P) \coloneqq \big\{ x\in X \colon P \cap \overline{\{T^n x \colon n\ge 0\}} = \emptyset\big\}.
\end{equation*}
When $P=\{p\}$ is a singleton we write $\mathcal E_T(p)$ instead of $\mathcal E_T(\{p\})$.

This article proves that, for a class of nonuniformly expanding interval maps with a neutral fixed point at \(0\), the set \(\mathcal E_f(P)\) is absolute winning for every countable \(P\subset[0,1]\). Specifically, we assume that $f \colon [0,1] \to [0,1]$ is a two-branch interval map and that there exists $r_1 \in (0,1)$ such that the following conditions hold (using one-sided derivatives as needed):
\begin{enumerate}[label=(MP\arabic*)]
	\item \label{condfirst} $f((0,r_1)) = f((r_1,1)) = (0,1)$;
	\item \label{cond2} $f$ is $C^1$ on $[0,r_1)$ and $C^2$ on $(0,r_1) \cup (r_1,1]$;
	\item \label{cond3} $f(0) = f(r_1) = 0$, $f'(0) = 1$, and $\lvert f' \rvert > 1$ on $(0,1]$
	      (hence $\inf_{x \in [\epsilon,1]} \, \lvert f' \rvert > 1$ for all $\epsilon > 0$);
	\item \label{cond4} $\sup_{x \in [0,1]} \, \lvert f'(x) \rvert < \infty$ and
	      $\sup_{x \in (r_1,1]} \, \lvert f''(x) \rvert < \infty$;
	\item \label{condlast} there exist $C \geq 1$ and $\gamma > 0$ such that for $x \in (0,r_1)$,
	      \begin{equation*}
		      C^{-1} \leq \frac{f''(x)}{x^{\gamma-1}} \leq C.
	      \end{equation*}
\end{enumerate}

\begin{remark}
	The value assigned to \(f\) at the branch point \(r_1\) is a convention. We take \(f(r_1)=0\) for definiteness. Changing this value affects only the countable set of points whose orbits eventually hit \(r_1\), and does not affect any absolute-winning conclusion below. All branch estimates are made on branch interiors or using one-sided branch extensions.
\end{remark}

Condition \ref{cond3} implies that $f$ is monotone on the uniformly expanding branch $(r_1,1]$. We allow either orientation on this branch. The only difference is which endpoint of the inducing base corresponds to long excursions near the neutral fixed point. Accordingly, define
\begin{equation}\label{eq:q}
	q=
	\begin{cases}
		r_1, & \text{if } f \restriction_{(r_1,1]} \text{ is increasing}, \\
		1,   & \text{if } f \restriction_{(r_1,1]} \text{ is decreasing}.
	\end{cases}
\end{equation}
The standard example from this class is the Manneville--Pomeau map
\begin{equation*}
	f(x)=x+x^{1+\gamma}\;(\bmod\,1).
\end{equation*}

For such $f$ define the first entrance/return time to the base $[r_1,1]$ by
\begin{equation*}
	\tau(x) \coloneqq \inf \{ n\ge 1 \colon f^n(x) \in [r_1,1] \},
\end{equation*}
using the convention that $\inf \emptyset = \infty$. When $x\in[r_1,1]$ this is the first return time. This quantity is finite for every point outside the countable collection of preimages of the branch endpoints. Define the induced first return map $F \colon [r_1,1] \to [r_1,1]$ by
\begin{equation*}
	F(x) \coloneqq f^{\tau(x)}(x)
\end{equation*}
whenever $\tau(x) < \infty$ and arbitrarily assign $F(x) \in \{r_1,1\}$ whenever $x \in [r_1,1]$ never returns to the induced domain. Condition \ref{cond3} implies that $F$ is uniformly expanding.

Our main theorem is the following. See Section \ref{sec:absolute} for the definition of the absolute game.

\begin{thm}\label{thm:main}
	Let $f$ satisfy the assumptions above. For every countable set $P \subset [0,1]$, the exceptional set $\mathcal{E}_f(P)$ is absolute winning in $[0,1]$.
\end{thm}

We first prove directly that $\mathcal E_F(q)$ is absolute winning, where $q$ is the endpoint in \eqref{eq:q}. This is the only part of the paper where the unbounded geometry of the countably branched induced map is handled directly. We then approximate $F$ by finite-branch expanding maps $F_K$ that agree with $F$ on the first $K-1$ branches. Hu--Li--Yu proved that for finite-branch piecewise $C^{1+\delta}$ expanding interval maps, every target-avoidance set is $1/2$-strong winning. On intervals, $1/2$-strong winning implies absolute winning; hence their theorem supplies absolute winning for the finite-branch approximants. Combining these approximants with the already-proved absolute winning of $\mathcal E_F(q)$ yields the following induced countable-target theorem.

\begin{thm}\label{thm:secondary}
	Let $f$ satisfy the assumptions above. For every countable set $Q \subset [r_1,1]$, the exceptional set $\mathcal{E}_F(Q)$ is absolute winning in $[r_1,1]$.
\end{thm}

Finally, quasisymmetric invariance of absolute winning transfers these induced results back to the original map $f$.

\section{The absolute game}\label{sec:absolute}

We recall McMullen's absolute game \cite{McMullen2010} specialized to the present setting. Let $I\subset\mathbb R$ be a nontrivial compact interval and let $S\subset I$. Fix $0<\beta<1/3$. In the $\beta$-absolute game, Bob chooses a nontrivial closed interval $B_1\subset I$. After Bob chooses $B_n$, Alice deletes a closed interval $A_n\subset B_n$ satisfying
\begin{equation*}
	\lvert A_n\rvert\le \beta \lvert B_n\rvert.
\end{equation*}
Bob then chooses a closed interval satisfying
\begin{equation*}
	B_{n+1}\subset B_n\setminus A_n,
	\qquad
	\lvert B_{n+1}\rvert\ge \beta \lvert B_n\rvert.
\end{equation*}
Alice wins if $\bigcap B_n$ meets $S$. If Alice is able to win regardless of Bob's choices, $S$ is called $\beta$-absolute winning; $S$ is called absolute winning if it is $\beta$-absolute winning for every $0<\beta<1/3$.

We will repeatedly use the following standard facts due to McMullen \cite{McMullen2010} without further mention.

\begin{thm}[McMullen]\label{thm:winningproperties}
	Absolute winning subsets of an interval are dense, uncountable, and have full Hausdorff dimension. The absolute winning property is stable under countable intersections and quasisymmetric homeomorphisms. Supersets and cocountable subsets of absolute winning sets are again absolute winning. Consequently, if $S,C \subseteq I$ and $C$ is countable, then $S\cup C$ is absolute winning in $I$ if and only if $S$ is absolute winning in $I$.
\end{thm}

In addition, Badziahin, Harrap, Nesharim, and Simmons proved the following relationship between the various notions of winning in one dimension \cite{BHNS}.

\begin{thm}[Badziahin, Harrap, Nesharim, and Simmons]
	A subset of $\mathbb{R}$ that is $1/2$-strong winning is absolute winning.
\end{thm}

We will also use the following finite-branch theorem of Hu--Li--Yu \cite{HLY}.

\begin{thm}[Hu--Li--Yu]\label{thm:HLY}
	Let $T\colon X\to X$ be a piecewise $C^{1+\delta}$ expanding interval map with finitely many branches on a compact interval $X$. Then, for every $\xi\in X$, the exceptional set
	\begin{equation*}
		\mathcal E_T(\xi)=\big\{ x\in X\colon\xi\notin\overline{\{T^n x \colon n\ge0\}}\big\}
	\end{equation*}
	is $1/2$-strong winning, hence absolute winning, in $X$.
\end{thm}

Finally, we record the following elementary locality principle to pass from absolute winning on the components of the complement of a closed countable set to absolute winning on the whole interval.

\begin{lem}[A locality property]\label{lem:locality}
	Let $S\subset[0,1]$. Suppose $C\subset[0,1]$ is closed and countable and every component $U$ of $[0,1]\setminus C$ has the following property: for every nontrivial compact interval $I\subset U$, the set $S\cap I$ is absolute winning in $I$. Then $S$ is absolute winning in $[0,1]$.
\end{lem}

\begin{proof}
	This is standard, but we include the short argument for completeness. Fix $0<\beta<1/3$. Alice begins by making arbitrary legal deletions of maximal length until Bob first chooses an interval contained in a single component $U$ of $[0,1]\setminus C$. If this never happens, then Bob's intervals shrink to a single point $\omega$. Since $C$ is closed, if $\omega\notin C$, then $\omega$ lies in some component $U$ of $[0,1]\setminus C$, and all sufficiently small Bob intervals would be contained in $U$, a contradiction. Hence $\omega\in C$, so Alice wins for the target $S\cup C$.

	Otherwise, suppose Bob chooses an interval $B_n$ contained in a component $U$ of $[0,1]\setminus C$. Choose a nontrivial compact interval $I\subset U$ containing $B_n$. From that point on, Alice plays the $\beta$-absolute winning strategy for $S\cap I$ inside $I$. Since all later Bob intervals remain inside $I$, this is also a legal strategy in the original game on $[0,1]$, and it forces $\bigcap B_n \subset S\cap I\subset S$. Thus $S\cup C$ is absolute winning in $[0,1]$ and the result follows.
\end{proof}

\section{The induced Markov structure}

All intervals are endowed with the subspace topology inherited from $\mathbb R$. We suppress the word ``relative'' when referring to open and closed subsets. All words/itineraries are assumed to be finite, and we allow the empty word $\varnothing$.

Define $\{r_n\}_{n\ge 0}\subset[0,1]$ recursively by
\begin{equation*}
	r_0\coloneqq1,
	\qquad
	f(r_{n+1})=r_n,
	\quad
	r_{n+1}\in(0,r_n).
\end{equation*}
Then $r_n\searrow 0$. Now define $\{p_n\}_{n\geq0} \subset [r_1,1]$ by
\begin{equation*}
	p_0\coloneqq
	\begin{cases}
		1,   & \text{if } f\restriction_{(r_1,1]} \text{ is increasing}, \\
		r_1, & \text{if } f\restriction_{(r_1,1]} \text{ is decreasing},
	\end{cases}
	\qquad
	f(p_n)=r_n \quad \text{for } n \geq 1.
\end{equation*}
If $f$ is increasing on $(r_1,1]$, then $p_n\searrow r_1$. If $f$ is decreasing on $(r_1,1]$, then $p_n\nearrow 1$.

For $k \geq 1$ let $J_k$ denote the closure of the first-return branch with return time $k$. Explicitly,
\begin{equation*}
	J_k \coloneqq
	\begin{cases}
		[p_k,p_{k-1}], & \text{if }f \restriction_{(r_1,1]}\text{ is increasing}, \\
		[p_{k-1},p_k], & \text{if }f \restriction_{(r_1,1]}\text{ is decreasing}.
	\end{cases}
\end{equation*}
The intervals $J_k$ form the first generation of the Markov partition for $F$, and $F$ maps the interior of each $J_k$ diffeomorphically onto $(r_1,1)$. In both orientation cases, the large-$k$ branches $J_k$ accumulate at the endpoint $q$ defined in \eqref{eq:q}.

Let $G_0 \coloneqq \{[r_1,1]\}$. For $n\ge 1$, $G_n$ is the family of closures of maximal intervals of monotonicity for $F^n$. If $\sigma=m_1\cdots m_n$ is a word in positive integers, write $J_\sigma\in G_n$ for the corresponding cylinder. Equivalently, $J_{\sigma k}$ is the child of $J_\sigma$ whose image under $F^n$ is $J_k$, up to endpoints. Thus the children $J_{\sigma k}$ accumulate at the unique endpoint
\begin{equation*}
	e_q(J_\sigma)
\end{equation*}
of $J_\sigma$ that maps to $q$ under the continuous extension of $F^n \restriction_{J_\sigma}$.

For $0<\theta<1$, define the dynamically oriented $q$-tail
\begin{equation}\label{eq:taudef}
	\Tau_{\theta}(J_\sigma) \coloneqq J_\sigma\cap B\bigl(e_q(J_\sigma),\theta \lvert J_\sigma\rvert\bigr).
\end{equation}
Equivalently, $\Tau_{\theta}(J_\sigma)$ is the subinterval of $J_\sigma$ of relative length $\theta$ that shares the endpoint $e_q(J_\sigma)$. In the increasing second-branch case these are the left tails. In the decreasing second-branch case the side alternates with the orientation of $F^n \restriction_{J_\sigma}$. The tail $\Tau_{\theta}(J_\sigma)$ contains all descendant cylinders $J_{\sigma k}$ for sufficiently large $k$.

The following estimates are the only geometric facts about the induced map needed below. They follow from Young's asymptotics near the neutral fixed point and a bounded distortion estimate for $F$; see \cite[Section 6.2]{Young1999} and \cite[Propositions 4.4 and 4.5]{DuvallMP}.

\begin{prop}[Geometry of the induced cylinders]\label{prop:geometry}
	There is a constant $C_1\ge 1$ such that for every word $\sigma$ and every $k\ge 1$,
	\begin{align}
		C_1^{-1}k^{-1/\gamma}
		 & \le
		\frac{\left\lvert\bigcup_{i=k}^\infty J_{\sigma i}\right\rvert}{\lvert J_\sigma\rvert}
		\le C_1 k^{-1/\gamma}, \label{eq:tail} \\
		C_1^{-1}k^{-1-1/\gamma}
		 & \le
		\frac{\lvert J_{\sigma k}\rvert}{\lvert J_\sigma\rvert}
		\le C_1 k^{-1-1/\gamma}.\label{eq:child}
	\end{align}
\end{prop}

\section{Absolute winning for the induced map}

In this section we prove the following result.

\begin{thm}\label{thm:Efqwinning}
	The exceptional set $\mathcal{E}_F(q)$ is absolute winning in the induced domain $[r_1,1]$.
\end{thm}

Membership in \(\mathcal E_F(q)\) is guaranteed by uniform avoidance of sufficiently deep \(q\)-tails, modulo endpoints, as the next lemma shows. For $K\ge 1$, set
\begin{equation*}
	U_K \coloneqq \{q\} \cup \bigcup_{i=K}^\infty J_i.
\end{equation*}
When $K>1$, its interior $U_K^\circ$ is a one-sided neighborhood of $q$ in the inducing base equal to $[r_1,p_{K-1})$ in the increasing second branch case and $(p_{K-1},1]$ in the decreasing case. Also define
\begin{equation*}
	\mathcal{C}_F \coloneqq \bigcup_{\lvert \sigma \rvert \geq0} \partial J_\sigma
\end{equation*}
to be the countable set of all endpoints of all cylinder sets (where $J_\varnothing \coloneqq [r_1,1]$).

\begin{lem}\label{lem:tail-to-exceptional}
	Fix $0<\theta<1$ and an integer $N \geq 0$. Suppose that $x\notin \Tau_{\theta}(J_\sigma)$ for all words $\sigma$ with $\lvert \sigma \rvert \geq N$. Then $x\in\EFq \cup \mathcal{C}_F$.
\end{lem}

\begin{proof}
	We may assume $x \notin \mathcal{C}_F$. By \eqref{eq:tail}, choose $K$ so large that
	\begin{equation*}
		\left\lvert\bigcup_{i=K}^\infty J_{\sigma i}\right\rvert\le \theta \lvert J_\sigma\rvert
	\end{equation*}
	for every word $\sigma$; this gives
	\begin{equation*}
		\bigcup_{i=K}^\infty J_{\sigma i}\subset \Tau_{\theta}(J_\sigma).
	\end{equation*}
	Now suppose $F^n(x)\in U_K^\circ$ for some $n \geq N$. Let $J_\sigma$ be the unique length-$n$ cylinder containing $x$. Then
	\begin{equation*}
		x\in\bigcup_{i=K}^\infty J_{\sigma i} \subset \Tau_{\theta}(J_\sigma),
	\end{equation*}
	a contradiction. Thus the points $\{F^n(x)\}_{n=N}^\infty$ miss the neighborhood $U_K^\circ$ of $q$, and since $x \notin \mathcal{C}_F$, the finite collection $\{F^n(x)\}_{n=0}^{N-1}$ misses $q$. Thus $x \in \mathcal{E}_F(q)$.
\end{proof}

The point requiring care in the proof of Theorem \ref{thm:Efqwinning} is that Bob's intervals can meet two cylinder tails of the same generation that Alice must avoid: one from each of two adjacent parent cylinders. In the decreasing case these tails may even meet at the common endpoint. Thus one cannot assert that there is a unique active tail at each step for Alice to avoid. We remove this obstruction by coloring the cylinders.

For each generation $g\ge 0$, fix a coloring
\begin{equation*}
	\chi_g \colon G_g\to\{0,1\}
\end{equation*}
such that adjacent elements of $G_g$ have different colors. If $I\in G_g$, write $\chi(I)=\chi_g(I)$.

For $\varepsilon\in\{0,1\}$ define
\begin{equation*}
	\begin{split}
		W_\varepsilon(q) \coloneqq \big\{x\in[r_1,1]\colon{} & \text{there exist }\theta>0\text{ and }N\ge0\text{ such that}              \\
		                                                     & x\notin \Tau_{\theta}(I)\text{ for every cylinder }I\in\bigcup_{g\ge N}G_g \\
		                                                     & \text{with }\chi(I)=\varepsilon\big\}.
	\end{split}
\end{equation*}
We will prove in Proposition \ref{prop:colored-winning} below that $W_\varepsilon(q)$ is absolute winning for $\varepsilon \in \{0,1\}$. Theorem \ref{thm:Efqwinning} will then follow easily.

The proof of Proposition \ref{prop:colored-winning} uses the notion of commensurability with the Markov partition, introduced in \cite{ManceTseng}. A proper nontrivial closed interval $B\subsetneq[r_1,1]$ is said to be \emph{commensurate with generation $g$}, abbreviated c.w.g.\ $g$, if $B$ contains an element of $G_g$ but no element of $G_{g-1}$. Every proper nontrivial closed interval is c.w.g.\ a unique generation. If $B$ is c.w.g.\ $g$, then $B$ intersects at most two elements of $G_{g-1}$. See \cite[Section 5]{DuvallMP} for proofs of these elementary facts.

Next we introduce the key notion of supported cylinders.

\begin{defn}[Supported cylinders]\label{def:supported-cylinder}
	Let $B$ be a closed interval, let $g\geq 1$, and fix
	$J_\sigma\in G_{g-1}$. Put
	\begin{equation*}
		S_B(J_\sigma)
		\coloneqq
		\{a\geq 1\colon J_{\sigma a}\subset B\}.
	\end{equation*}
	If $S_B(J_\sigma)=\emptyset$, define $\mathcal A_B(J_\sigma)\coloneqq\emptyset$; otherwise set
	\begin{equation*}
		a_0(J_\sigma)\coloneqq \min S_B(J_\sigma),
		\qquad
		b_0(J_\sigma)\coloneqq \max\{1,a_0(J_\sigma)-1\},
	\end{equation*}
	and define the support tail of $J_\sigma$ by
	\[
		\mathcal A_B(J_\sigma)
		\coloneqq
		\{J_{\sigma a}\colon a\geq b_0(J_\sigma)\}.
	\]

	Now let $J\in G_h$ with $h\geq g$, and write
	\[
		J=J_{\sigma a\rho},
		\qquad
		\lvert\sigma\rvert=g-1,
		\qquad
		\lvert\rho\rvert=h-g.
	\]
	We say that \emph{$J$ is supported by $B$ at generation $g$} if
	\[
		J_{\sigma a}\in \mathcal A_B(J_\sigma).
	\]
	Equivalently, the generation-\(g\) ancestor of \(J\) is either one of the children \(J_{\sigma a}\) with \(a\ge a_0\), or, when \(a_0>1\), the immediate predecessor \(J_{\sigma(a_0-1)}\), which necessarily meets $B$.
\end{defn}

The purpose of this definition is size control: a supported cylinder need not meet Bob's interval \(B\), but its generation-\(g\) ancestor is nevertheless comparable in size to the first child of \(J_\sigma\) contained in \(B\).

\begin{lem}[Local support criterion]\label{lem:local-support-criterion}
	Let $B$ be a closed interval, let $g\geq1$, and let $J\in G_h$ with $h\geq g$. Write
	\[
		J=J_{\sigma a\rho},
		\qquad
		\lvert\sigma\rvert=g-1,
		\qquad
		\lvert\rho\rvert=h-g.
	\]
	Assume that
	\[
		B\cap J\setminus\mathcal{C}_F\neq\emptyset
	\]
	and that
	\[
		S_B(J_\sigma)\neq\emptyset.
	\]
	Then $J$ is supported by $B$ at generation $g$.
\end{lem}

\begin{proof}
	Put
	\[
		a_0\coloneqq a_0(J_\sigma),
		\qquad
		b_0\coloneqq b_0(J_\sigma).
	\]
	By definition,
	\[
		J_{\sigma a_0}\subset B.
	\]
	Since $B\cap J\setminus\mathcal{C}_F\neq\emptyset$, the interval $B$ meets the interior of $J$. In particular, $B$ meets the interior of the generation-$g$ ancestor $J_{\sigma a}$.

	We claim that $a\geq b_0$. If $a_0=1$, then $b_0=1$, and this is immediate. Suppose $a_0\geq2$, so that $b_0=a_0-1$. If $a<a_0-1$, then the connected interval $B\cap J_\sigma$ meets $J_{\sigma a}$ and contains $J_{\sigma a_0}$. Therefore it contains the intervening child $J_{\sigma(a+1)}$. But then
	\[
		a+1<a_0
	\]
	and
	\[
		J_{\sigma(a+1)}\subset B,
	\]
	contradicting the minimality of $a_0$.

	Thus $a\geq b_0$, and hence
	\[
		J_{\sigma a}\in\mathcal A_B(J_\sigma).\qedhere
	\]
\end{proof}

The next lemma identifies the only ways in which Bob's interval can meet a deep tail: either the relevant cylinder is supported by Bob's previous interval, or Bob's interval is already large compared with that cylinder.

For any cylinder \(J=J_\sigma\), recall that \(e_q(J)\) is the endpoint of \(J\) that is mapped to \(q\) by the continuous extension of \(F^{\lvert \sigma\rvert}\restriction_{J_\sigma}\). Define
\begin{equation*}
	e_{\mathrm{opp}}(J)
\end{equation*}
to be the endpoint of \(J\) opposite \(e_q(J)\).

\begin{lem}[Support-or-size criterion]\label{lem:support-or-size}
	Let $B$ be a closed interval that is c.w.g.\ $g\geq1$, let $0<\theta<1$, and let $J\in G_h$ with $h\geq g$. If
	\[
		B\cap\Tau_\theta(J)\setminus\mathcal{C}_F\neq\emptyset,
	\]
	then one of the following conclusions holds.
	\begin{enumerate}[label=(\roman*)]
		\item\label{it:support-or-size-supported} There exists an integer $s$ with $g\leq s\leq h$ such that $J$ is supported by $B$ at generation $s$.
		\item\label{it:support-or-size-large} One has
		      \[
			      \lvert B\rvert\geq (1-\theta)\lvert J\rvert.
		      \]
	\end{enumerate}
\end{lem}

\begin{proof}
	Write the ancestor chain of $J$ from generation $g-1$ to generation $h$ as
	\[
		J_{\eta_{g-1}}\supset J_{\eta_g}\supset\cdots\supset J_{\eta_h}=J,
		\qquad
		J_{\eta_s}=J_{\eta_{s-1}a_s}.
	\]
	If \(S_B(J_{\eta_s})\neq\emptyset\) for some \(g-1\le s\le h-1\), then Lemma~\ref{lem:local-support-criterion}, applied with generation \(s+1\), shows that \(J\) is supported by \(B\) at generation \(s+1\). Since \(g\le s+1\le h\), alternative (i) holds. Thus suppose
	\begin{equation*}
		S_B(J_{\eta_s}) = \emptyset
	\end{equation*}
	for all $g-1 \leq s \leq h-1$. Observe that $B$ meets the interior of each cylinder set in the ancestor chain of $J$ since $B \cap \Tau_\theta(J) \setminus \mathcal{C}_F \neq \emptyset$.

	We now use induction to show that
	\[
		e_{\mathrm{opp}}(J_{\eta_s})\in B
	\]
	for every $g-1\leq s\leq h$.  For the initial step, since $B$ is c.w.g.\ $g$, it contains some element of $G_g$. This element cannot be a child of $J_{\eta_{g-1}}$, because $S_B(J_{\eta_{g-1}})=\emptyset$. Since $B$ also meets the interior of $J_{\eta_{g-1}}$, it follows that $B$ contains an endpoint of $J_{\eta_{g-1}}$. This endpoint cannot be $e_q(J_{\eta_{g-1}})$, because then $B$ would contain arbitrarily large-index children of $J_{\eta_{g-1}}$, contradicting $S_B(J_{\eta_{g-1}})=\emptyset$. Hence
	\[
		e_{\mathrm{opp}}(J_{\eta_{g-1}})\in B.
	\]

	Now fix $t$ with $g\leq t\leq h$, and suppose that
	\[
		e_{\mathrm{opp}}(J_{\eta_{t-1}})\in B.
	\]
	Then since $B$ meets the interior of $J_{\eta_{t-1}}$, $B$ also meets $J_{\eta_{t-1}1}$. Also
	\begin{equation*}
		e_{\mathrm{opp}}(J_{\eta_{t-1}1}) = e_{\mathrm{opp}}(J_{\eta_{t-1}})\in B.
	\end{equation*}
	If $B$ also meets $J_{\eta_{t-1}2}$, then $B$ contains the first child $J_{\eta_{t-1}1}$, contradicting
	\[
		S_B(J_{\eta_{t-1}})=\emptyset.
	\]
	Therefore $a_t=1$ and hence
	\[
		e_{\mathrm{opp}}(J_{\eta_t})=e_{\mathrm{opp}}(J_{\eta_{t-1}})\in B,
	\]
	completing the induction.

	In particular,
	\[
		e_{\mathrm{opp}}(J)\in B.
	\]
	Since $B$ also meets the tail $\Tau_\theta(J)$, $B$ contains the interval $J \setminus \Tau_\theta(J)$. Hence
	\[
		\lvert B\rvert\geq \lvert J \rvert - \lvert \Tau_\theta(J) \rvert = (1-\theta)\lvert J\rvert.\qedhere
	\]
\end{proof}

\begin{lem}[Enlarged tail estimate]\label{lem:enlarged-tail-estimate}
	Let $M>4^{1/\gamma}C_1$ and set $\theta=M^{-1}$. Let $J=J_\tau$, and choose $K\geq1$ so that
	\[
		\bigcup_{i=K+1}^{\infty}J_{\tau i}
		\subset
		\Tau_\theta(J)
		\subset
		\bigcup_{i=K}^{\infty}J_{\tau i}.
	\]
	Then $K\geq2$ and
	\begin{equation}\label{eq:enlarged-tail-estimate}
		\frac{\left\lvert\bigcup_{i=K-1}^{\infty}J_{\tau i}\right\rvert}{\lvert J\rvert}
		\leq
		2^{1/\gamma}C_1^2M^{-1}.
	\end{equation}
\end{lem}

\begin{proof}
	Using the lower bound in \eqref{eq:tail} together with the first inclusion above, we obtain
	\[
		C_1^{-1}(K+1)^{-1/\gamma}\leq M^{-1}.
	\]
	Hence
	\[
		K+1\geq C_1^{-\gamma}M^\gamma.
	\]
	Since $M>4^{1/\gamma}C_1$, this implies
	\[
		K-1\geq \frac{1}{2}C_1^{-\gamma}M^\gamma\geq1.
	\]
	Therefore, using the upper bound in \eqref{eq:tail},
	\[
		\frac{\left\lvert\bigcup_{i=K-1}^{\infty}J_{\tau i}\right\rvert}{\lvert J\rvert}
		\leq
		2^{1/\gamma}C_1^2M^{-1}.\qedhere
	\]
\end{proof}

\begin{lem}[Tail crossing criterion]\label{lem:tail-crossing-criterion}
	Let $J=J_\tau$, and let $K\geq2$ satisfy
	\[
		\Tau_\theta(J)
		\subset
		\bigcup_{i=K}^{\infty}J_{\tau i}.
	\]
	Let $B$ be a closed interval such that
	\[
		B\cap\Tau_\theta(J)\setminus\mathcal{C}_F\neq\emptyset
	\]
	and
	\[
		\lvert B\rvert>
		\left\lvert\bigcup_{i=K-1}^{\infty}J_{\tau i}\right\rvert.
	\]
	Then $B$ contains a child of $J$.
\end{lem}

\begin{proof}
	If $e_q(J)\in B$, then, since $B\cap\Tau_\theta(J)$ contains a point not belonging to $\mathcal{C}_F$, the interval $B$ contains $J_{\tau i}$ for all sufficiently large $i$, and we are done. Thus suppose $e_q(J)\notin B$.

	Then both endpoints of $B$ lie on the same side of $e_q(J)$. Let $e_1(B)$ be the endpoint of $B$ closer to $e_q(J)$, and let $e_2(B)$ be the other endpoint. Since $B$ meets $\Tau_\theta(J)$ and
	\[
		\Tau_\theta(J)
		\subset
		\bigcup_{i=K}^{\infty}J_{\tau i},
	\]
	we have
	\[
		e_1(B)\in\bigcup_{i=K}^{\infty}J_{\tau i}.
	\]
	The endpoint $e_2(B)$ cannot lie in
	\[
		\bigcup_{i=K-1}^{\infty}J_{\tau i},
	\]
	for otherwise $B$ would be contained in this enlarged tail, contradicting the assumed length inequality. Hence $e_2(B)$ lies on the opposite side of $J_{\tau(K-1)}$ from $e_1(B)$. Therefore $B$ contains $J_{\tau(K-1)}$.
\end{proof}

The next lemma is the key estimate used in Proposition \ref{prop:colored-winning} below: supported cylinders are never too large relative to Bob's interval, and tail intersections force the next generation to appear inside Bob's interval.

\begin{lem}[Supported c.w.g.\ estimate]\label{lem:supported-cwg}
	Fix $0<\beta<1/3$. Let $M>1$ satisfy
	\begin{equation}\label{eq:M-supported-cwg}
		M>
		\max\{
		2^{1+1/\gamma}C_1^2\beta^{-2},\,
		2^{1+2/\gamma}C_1^4\beta^{-1}
		\}.
	\end{equation}
	Set
	\[
		\theta=M^{-1}.
	\]
	Suppose $B\subset B^{-}$ are closed intervals with
	\[
		\lvert B\rvert\geq \beta\lvert B^{-}\rvert
	\]
	and that $B$ is c.w.g.\ $g'\geq2$.
	\begin{enumerate}[label=(\roman*)]
		\item\label{it:supported-size} If $J\in G_h$ and $J$ is supported by $B^{-}$ at a generation $s$ with $1\leq s\leq h$, then
		      \[
			      \lvert B\rvert \geq 2^{-1-1/\gamma}\beta C_1^{-2}\lvert J\rvert.
		      \]

		\item\label{it:supported-crossing} If $J \in G_h$ with $1\leq h\leq g'-2$, and $J$ is supported by $B^{-}$ at a generation $s$ with $1\leq s\leq h$, and if $B\cap\Tau_\theta(J)$ contains a point not in $\mathcal{C}_F$, then $B$ contains an element of $G_{h+1}$.
	\end{enumerate}
\end{lem}

\begin{proof}
	For \ref{it:supported-size}, write
	\[
		J=J_{\sigma a\rho},
		\qquad
		\lvert\sigma\rvert=s-1,
		\qquad
		\lvert\rho\rvert=h-s.
	\]
	By Definition~\ref{def:supported-cylinder},
	\[
		J_{\sigma a}\in \mathcal A_{B^{-}}(J_\sigma).
	\]
	In particular, $S_{B^{-}}(J_\sigma)\neq\emptyset$. Put
	\[
		a_0\coloneqq a_0(J_\sigma),
		\qquad
		b_0\coloneqq b_0(J_\sigma).
	\]
	By the definition of $a_0$,
	\[
		J_{\sigma a_0}\subset B^{-}.
	\]
	By the definition of $\mathcal A_{B^{-}}(J_\sigma)$, the child index $a$ satisfies
	\[
		a\geq b_0.
	\]
	Since either \(b_0=a_0=1\) or \(b_0=a_0-1\), we have
	\[
		\frac{a_0}{a}\leq \frac{a_0}{b_0}\leq 2 .
	\]
	By the child estimate \eqref{eq:child}, applied to the two children \(J_{\sigma a_0}\) and \(J_{\sigma a}\), we get
	\[
		\lvert J_{\sigma a_0}\rvert
		\geq C_1^{-1}a_0^{-1-1/\gamma}\lvert J_\sigma\rvert
	\]
	and
	\[
		\lvert J_{\sigma a}\rvert
		\leq C_1a^{-1-1/\gamma}\lvert J_\sigma\rvert .
	\]
	Therefore
	\[
		\frac{\lvert J_{\sigma a_0}\rvert}{\lvert J_{\sigma a}\rvert}
		\geq
		C_1^{-2}\left(\frac{a_0}{a}\right)^{-1-1/\gamma}
		\geq
		2^{-1-1/\gamma}C_1^{-2}.
	\]
	Since \(J_{\sigma a_0}\subset B^{-}\) and \(\lvert B\rvert\geq \beta\lvert B^{-}\rvert\), it follows that
	\[
		\lvert B\rvert
		\geq
		\beta\lvert J_{\sigma a_0}\rvert
		\geq
		2^{-1-1/\gamma} C_1^{-2} \beta \lvert J_{\sigma a}\rvert .
	\]
	Finally, since
	\[
		J=J_{\sigma a\rho}\subset J_{\sigma a},
	\]
	we obtain
	\[
		\lvert B\rvert
		\geq
		2^{-1-1/\gamma}\beta C_1^{-2}\lvert J\rvert.
	\]

	For \ref{it:supported-crossing}, write $J=J_\tau$. Choose $K\geq1$ so that
	\[
		\bigcup_{i=K+1}^{\infty}J_{\tau i}
		\subset
		\Tau_\theta(J)
		\subset
		\bigcup_{i=K}^{\infty}J_{\tau i}.
	\]
	By Lemma~\ref{lem:enlarged-tail-estimate},
	\[
		\left\lvert\bigcup_{i=K-1}^{\infty}J_{\tau i}\right\rvert
		\leq
		2^{1/\gamma}C_1^2M^{-1}\lvert J\rvert.
	\]
	By \ref{it:supported-size},
	\[
		\lvert B\rvert
		\geq
		2^{-1-1/\gamma}\beta C_1^{-2}\lvert J\rvert.
	\]
	Since
	\[
		M>2^{1+2/\gamma}C_1^4\beta^{-1},
	\]
	we have
	\[
		2^{-1-1/\gamma}\beta C_1^{-2}
		>
			2^{1/\gamma}C_1^2M^{-1}.
	\]
	Therefore
	\[
		\lvert B\rvert
		>
		\left\lvert\bigcup_{i=K-1}^{\infty}J_{\tau i}\right\rvert.
	\]
	Note that
	\begin{equation*}
		M > 2^{1+2/\gamma} C_1^4 \beta^{-1} > 4^{1/\gamma}C_1
	\end{equation*}
	so that Lemma~\ref{lem:tail-crossing-criterion} now implies that $B$ contains a child of $J$, which is an element of $G_{h+1}$.
\end{proof}

We are ready to prove that Alice has a winning strategy for the absolute game played on the color-sensitive exceptional sets $W_\varepsilon(q)$ defined at the start of this section.

\begin{prop}\label{prop:colored-winning}
	For each $\varepsilon\in\{0,1\}$, the set $W_\varepsilon(q)$ is absolute winning in $[r_1,1]$.
\end{prop}

\begin{proof}
	We will target the set $W_\varepsilon(q) \cup \mathcal{C}_F$. Once we show it is absolute winning, the result follows immediately.
	
	Fix $0<\beta<1/3$ and $\varepsilon\in\{0,1\}$. We describe Alice's $\beta$-absolute strategy by induction. First we state the induction hypotheses. They are formulated at times immediately after Bob has chosen $B_{n_j}$, before Alice has responded.

	\medskip
	\noindent\textbf{The induction hypothesis.}
	Suppose there exist $\theta>0$, an integer $k\geq2$, and integers
	\begin{equation*}
		1\leq n_1<n_2<\cdots<n_{k-1}
	\end{equation*}
	and
	\begin{equation*}
		1\leq g_1<g_2<\cdots<g_{k-1}
	\end{equation*}
	such that the players have chosen $\{ B_i \}_{i=1}^{n_{k-1}}$ and $\{ A_i \}_{i=1}^{n_{k-1}-1}$, where all of Alice's deleted intervals $A_i$ are of maximal length $\lvert A_i \rvert = \beta \lvert B_i \rvert$. Suppose further that for all $1\leq j \leq k-1$, Bob's interval $B_{n_j}$ satisfies:
	\begin{enumerate}[label=$P_{\arabic*}(j)$:]
		\item\label{it:P1} $B_{n_j}$ is c.w.g.\ $g_j$;
		\item\label{it:P2} For every $\varepsilon$-colored cylinder $I\in G_{g_j-1}$,
		      \begin{equation*}
			      \lvert B_{n_j}\cap\Tau_\theta(I)\rvert\le \beta\lvert B_{n_j}\rvert;
		      \end{equation*}
		\item\label{it:P3} For every $\varepsilon$-colored cylinder $I\in\bigcup_{g=g_1-1}^{g_j-2}G_g$,
		      \begin{equation*}
			      B_{n_j}\cap \Tau_{\theta}(I) \subset \mathcal{C}_F.
		      \end{equation*}
	\end{enumerate}

	\medskip
	\noindent\textbf{The initial step.}
	Bob chooses $B_1$. If $B_1 \neq [r_1,1]$ define $n_1 \coloneqq 1$. If $B_1=[r_1,1]$, define $n_1 \coloneqq 2$; Alice deletes the concentric subinterval of $B_1$ of length $\beta \lvert B_1 \rvert$ and waits for Bob's next interval $B_2$. In either case $B_{n_1}$ is c.w.g.\ $g_1 \geq 1$, confirming $P_1(1)$.

	Choose $M'>1$ so large that
	\begin{equation}\label{eq:init-D0}
		\lvert I\rvert
		\leq
		M'\lvert B_{n_1}\rvert
	\end{equation}
	for every $I\in G_{g_1-1}$ intersecting $B_{n_1}$. This is possible since $B_{n_1}$ intersects at most two elements of $G_{g_1-1}$. Next, choose $M>1$ satisfying
	\begin{equation}\label{eq:d2choice}
		M
		>
		\max\{
		\beta^{-1}M',\,
		2^{1+1/\gamma}C_1^4\beta^{-2},\,
		2^{1+2/\gamma}C_1^6\beta^{-1},\,
		4^{1/\gamma}C_1
		\}.
	\end{equation}
	Finally set
	\begin{equation}\label{eq:thetachoice}
		\theta \coloneqq M^{-1}.
	\end{equation}
	Observe that for any $\varepsilon$-colored cylinder $I\in G_{g_1-1}$, \eqref{eq:init-D0} and \eqref{eq:d2choice} give
	\[
		\lvert B_{n_1}\cap\Tau_\theta(I)\rvert
		\leq
		\lvert\Tau_\theta(I)\rvert
		=
		M^{-1}\lvert I\rvert
		\leq
		M^{-1} M'\lvert B_{n_1}\rvert
		\leq
		\beta\lvert B_{n_1}\rvert,
	\]
	proving $P_2(1)$. The statement $P_3(1)$ is vacuously true.

	\medskip
	\noindent\textbf{The induction step.}
	Consider the active collection
	\begin{equation*}
		\mathcal I \coloneqq \left\{I\in G_{g_{k-1}-1} \colon \chi(I)=\varepsilon\text{ and } B_{n_{k-1}} \cap \Tau_{\theta}(I)\ne\emptyset\right\}.
	\end{equation*}
	Since $B_{n_{k-1}}$ is c.w.g.\ $g_{k-1}$, it intersects at most two elements of $G_{g_{k-1}-1}$. If it intersects two, they are adjacent and hence have different colors. Therefore
	\begin{equation}\label{eq:one-active-tail}
		\#\mathcal I\le1.
	\end{equation}

	If $\mathcal I=\emptyset$, Alice deletes the concentric subinterval $A_{n_{k-1}}$ of $B_{n_{k-1}}$ of length $\beta \lvert B_{n_{k-1}}\rvert$. If $\mathcal I=\{I\}$, Alice deletes any subinterval of $B_{n_{k-1}}$ of length $\beta \lvert B_{n_{k-1}} \rvert$ containing
	\begin{equation*}
		B_{n_{k-1}}\cap\Tau_{\theta}(I).
	\end{equation*}
	This is possible and legal by $P_2(k-1)$.

	Bob now chooses $B_{n_{k-1}+1}\subset B_{n_{k-1}}\setminus A_{n_{k-1}}$. If $B_{n_{k-1}+1}$ is c.w.g.\ $g>g_{k-1}$, set $n_k \coloneqq n_{k-1}+1$ and $g_k \coloneqq g$. If not, Alice plays the maximal middle-deletion move on each subsequent turn until Bob first chooses an interval $B_{n_k}$ that is commensurate with some generation $g_k > g_{k-1}$.

	Suppose for the moment that this never happens. Since Alice deletes intervals of maximal allowed length on every turn, Bob's intervals shrink to a single point; write $\bigcap_n B_n=\{\omega\}$. If $\omega\notin\mathcal{C}_F$, then some sufficiently high generation cylinder contains $\omega$ in its interior. All sufficiently late Bob intervals are contained in that cylinder, and hence are c.w.g.\ some generation larger than $g_{k-1}$, a contradiction. Thus $\omega\in\mathcal{C}_F\subset W_\varepsilon(q)\cup\mathcal{C}_F$, and Alice wins. Therefore we may disregard this case.

	\medskip
	\noindent\textbf{Propagation of the induction hypothesis.}
	We now check that $B_{n_k}$ satisfies $P_1(k)$--$P_3(k)$.

	The first condition $P_1(k)$ is true by construction. Set
	\[
		B^{-}\coloneqq B_{n_k-1},
		\qquad
		g\coloneqq g_{k-1}.
	\]
	By the minimal choice of $n_k$, the interval $B^{-}$ is c.w.g.\ $g$, while $B_{n_k}$ is c.w.g.\ $g_k>g$. Also
	\[
		B_{n_k}\subset B^{-},
		\qquad
		\lvert B_{n_k}\rvert\geq \beta\lvert B^{-}\rvert.
	\]

	For $P_2(k)$, let $J\in G_{g_k-1}$ be $\varepsilon$-colored. If
	\[
		B_{n_k}\cap\Tau_\theta(J)\subset\mathcal{C}_F,
	\]
	then $P_2(k)$ is immediate. Otherwise
	\[
		B_{n_k}\cap\Tau_\theta(J)\setminus\mathcal{C}_F\neq\emptyset.
	\]
	Since $B_{n_k}\subset B^{-}$, Lemma~\ref{lem:support-or-size}, applied to $B=B^{-}$, gives two alternatives.

	Suppose that the first alternative in Lemma~\ref{lem:support-or-size} holds: that there is an integer $s$ with
	\[
		g\leq s\leq g_k-1
	\]
	such that $J$ is supported by $B^{-}$ at generation $s$. Lemma~\ref{lem:supported-cwg}, applied with $B=B_{n_k}$, $g'=g_k$, and this value of $s$, gives
	\begin{equation*}
		\lvert B_{n_k}\rvert
		\geq
		2^{-1-1/\gamma}\beta C_1^{-2}\lvert J\rvert.
	\end{equation*}
	Since $M>2^{1+1/\gamma}C_1^2\beta^{-2}$ and $\theta=M^{-1}$, it follows that
	\[
		\lvert B_{n_k}\cap\Tau_\theta(J)\rvert
		\leq
		\theta\lvert J\rvert
		=
		M^{-1}\lvert J\rvert
		<
		\beta\lvert B_{n_k}\rvert.
	\]

	Now suppose instead that the second alternative in Lemma~\ref{lem:support-or-size} holds. Then
	\[
		\lvert B^{-}\rvert\geq (1-\theta)\lvert J\rvert,
	\]
	and hence
	\[
		\lvert B_{n_k}\rvert\geq \beta(1-\theta)\lvert J\rvert.
	\]
	Since $M>2\beta^{-2}$ and $\theta=M^{-1}$, we have $\theta<\beta^2(1-\theta)$. Therefore
	\[
		\lvert B_{n_k}\cap\Tau_\theta(J)\rvert
		\leq
		\theta\lvert J\rvert
		<
		\beta^2(1-\theta)\lvert J\rvert
		\leq
		\beta\lvert B_{n_k}\rvert.
	\]
	Thus $P_2(k)$ holds.

	It remains to prove $P_3(k)$. Let
	\begin{equation*}
		J\in\bigcup_{h=g_1-1}^{g_k-2}G_h
	\end{equation*}
	be an $\varepsilon$-colored cylinder. We must show
	\begin{equation*}
		B_{n_k}\cap\Tau_{\theta}(J)\subset\mathcal{C}_F.
	\end{equation*}
	There are three cases.

	First suppose $J$ has generation at most $g_{k-1}-2$. Then the conclusion follows from $P_3(k-1)$ and nesting, because $B_{n_k}\subset B_{n_{k-1}}$.

	Second suppose $J\in G_{g_{k-1}-1}$. If $B_{n_{k-1}}$ did not meet $\Tau_{\theta}(J)$, then nesting again gives the conclusion. If $B_{n_{k-1}}$ did meet $\Tau_{\theta}(J)$, then $J$ was the unique active $\varepsilon$-colored cylinder at stage $k-1$, by \eqref{eq:one-active-tail}. Alice deleted an interval containing
	\begin{equation*}
		B_{n_{k-1}}\cap\Tau_{\theta}(J),
	\end{equation*}
	so all later Bob intervals, including $B_{n_k}$, are disjoint from this tail.

	Third suppose $J\in G_h$ with
	\begin{equation*}
		g_{k-1}\le h\le g_k-2.
	\end{equation*}
	Suppose, toward a contradiction, that
	\[
		B_{n_k}\cap\Tau_\theta(J)\setminus\mathcal{C}_F\neq\emptyset.
	\]
	Since $B_{n_k}\subset B^{-}$, Lemma~\ref{lem:support-or-size}, applied to $B=B^{-}$, gives two alternatives.

	If there is an integer $s$ with
	\[
		g\leq s\leq h
	\]
	such that $J$ is supported by $B^{-}$ at generation $s$, then Lemma~\ref{lem:supported-cwg}, applied with $B=B_{n_k}$, $g'=g_k$, and this value of $s$, implies that $B_{n_k}$ contains an element of $G_{h+1}$.

	If instead
	\[
		\lvert B^{-}\rvert\geq (1-\theta)\lvert J\rvert,
	\]
	then
	\[
		\lvert B_{n_k}\rvert\geq \beta(1-\theta)\lvert J\rvert.
	\]
	Write $J=J_\tau$, and choose $K\geq1$ so that
	\[
		\bigcup_{i=K+1}^{\infty}J_{\tau i}
		\subset
		\Tau_\theta(J)
		\subset
		\bigcup_{i=K}^{\infty}J_{\tau i}.
	\]
	By Lemma~\ref{lem:enlarged-tail-estimate},
	\[
		\left\lvert\bigcup_{i=K-1}^{\infty}J_{\tau i}\right\rvert
		\leq
		2^{1/\gamma}C_1^2M^{-1}\lvert J\rvert.
	\]
	Since $M>2^{1+2/\gamma}C_1^4\beta^{-1}$, while $C_1\geq1$ and $\theta<1/2$, we have
	\[
		2^{1/\gamma}C_1^2M^{-1}
		<
		\beta(1-\theta).
	\]
	Therefore
	\[
		\lvert B_{n_k}\rvert
		>
		\left\lvert\bigcup_{i=K-1}^{\infty}J_{\tau i}\right\rvert.
	\]
	Lemma~\ref{lem:tail-crossing-criterion} implies that $B_{n_k}$ contains a child of $J$, again an element of $G_{h+1}$.

	In either alternative, $B_{n_k}$ contains an element of $G_{h+1}$. Since $h+1\le g_k-1$, this contained cylinder contains an element of $G_{g_k-1}$, contradicting that $B_{n_k}$ is c.w.g.\ $g_k$. Therefore
	\[
		B_{n_k}\cap\Tau_\theta(J)\setminus\mathcal{C}_F=\emptyset,
	\]
	and hence
	\begin{equation*}
		B_{n_k}\cap\Tau_\theta(J)\subset\mathcal{C}_F.
	\end{equation*}
	This proves $P_3(k)$ and completes the induction step.

	\medskip
	\noindent\textbf{The outcome of the game.}
	Because Alice deletes intervals of maximal allowed length on every turn, Bob's intervals shrink to a single point. Write
	\begin{equation*}
		\{\omega\}=\bigcap_{n=1}^{\infty}B_n.
	\end{equation*}
	If $\omega \in \mathcal{C}_F \subset W_\varepsilon(q) \cup \mathcal{C}_F$ then Alice wins, so assume $\omega \notin \mathcal{C}_F$.
	
	By construction, the stopping generations $g_j$ tend to infinity. Given any $\varepsilon$-colored cylinder $I$ of generation at least $g_1-1$, choose $j$ so large that
	\begin{equation*}
		I\in\bigcup_{h=g_1-1}^{g_j-2}G_h.
	\end{equation*}
	Then $P_3(j)$ gives
	\begin{equation*}
		B_{n_j}\cap\Tau_\theta(I)\subset\mathcal C_F.
	\end{equation*}
	Now $\omega \notin \mathcal{C}_F$ implies $\omega \notin B_{n_j}\cap\Tau_\theta(I)$. But $\omega \in B_{n_j}$ and thus
	\begin{equation*}
		\omega\notin\Tau_\theta(I).
	\end{equation*}
	Therefore
	\begin{equation*}
		\omega\in W_\varepsilon(q) \subset W_\varepsilon(q)\cup\mathcal{C}_F
	\end{equation*}
	and so Alice wins the $\beta$-absolute game for the target set $W_\varepsilon(q) \cup \mathcal{C}_F$. Since $\beta<1/3$ was arbitrary, $W_\varepsilon(q)\cup\mathcal{C}_F$, and hence $W_\varepsilon(q)$, is absolute winning.
\end{proof}

With this we can now prove the main theorem of this section: $\mathcal{E}_F(q)$ is absolute winning in $[r_1,1]$.

\begin{proof}[Proof of Theorem \ref{thm:Efqwinning}]
	By Proposition \ref{prop:colored-winning}, both $W_0(q)$ and $W_1(q)$ are absolute winning. Hence their intersection is absolute winning. If $x$ lies in this intersection then $x$ avoids all sufficiently deep $q$-tails with some common positive parameter $\theta$. Thus Lemma \ref{lem:tail-to-exceptional} gives
	\begin{equation*}
		W_0(q)\cap W_1(q)\subset\EFq\cup\mathcal{C}_F.
	\end{equation*}
	Therefore $\EFq\cup\mathcal{C}_F$, and hence $\EFq$, is absolute winning.
\end{proof}

\section{All targets for the induced map}\label{sec:all-targets-F}

We now use the finite-branch theorem of Hu--Li--Yu to pass from the distinguished endpoint target $q$ to arbitrary targets in the inducing base.

Let
\begin{equation*}
	Y\coloneqq [r_1,1]
\end{equation*}
be the inducing base. For $K\ge1$, recall
\begin{equation*}
	U_K=\{q\}\cup\bigcup_{i=K}^\infty J_i.
\end{equation*}
Thus $U_K^\circ$ is the one-sided tail neighborhood of $q$ in $Y$. When $K \geq 2$, $Y\setminus U_K^\circ$ is the union of finitely many first-generation branches.

For each $K\ge 2$, define a finite-branch map
\begin{equation*}
	F_K:Y\to Y
\end{equation*}
as follows. On each first-generation branch $J_i$ with $i<K$, set $F_K=F$. On the tail interval $U_K^\circ \setminus\{q\}$, let $F_K$ be an affine homeomorphism from $U_K^\circ$ onto $(r_1,1)$. The affine branch may be chosen in either orientation; only expansion and finite branching are relevant. Finally, set $F_K(q) = F(q)$.

Since $F$ is uniformly expanding and the affine branch of $F_K$ has constant slope $\lvert Y\rvert/\lvert U_K\rvert>1$, each $F_K$ is a finite-branch piecewise $C^2$ expanding interval map. Therefore Theorem \ref{thm:HLY} gives the following.

\begin{lem}\label{lem:FK-absolute}
	For every $K\ge 2$ and every $\xi\in Y$, the set $\mathcal E_{F_K}(\xi)$ is absolute winning in $Y$.
\end{lem}

We now show how these finite-branch approximants allow us to extend Theorem \ref{thm:Efqwinning} to all target points in $Y$.

\begin{thm}\label{thm:all-single-targets-F}
	For every $\xi\in Y$, the set $\mathcal{E}_F(\xi)$ is absolute winning in $Y$.
\end{thm}

\begin{proof}
	Fix $\xi\in Y$. By Theorem \ref{thm:Efqwinning}, the set $\EFq$ is absolute winning. By Lemma \ref{lem:FK-absolute}, each $\mathcal E_{F_K}(\xi)$ is absolute winning. Hence
	\begin{equation*}
		S\coloneqq \EFq\cap\bigcap_{K=2}^\infty \mathcal E_{F_K}(\xi)
	\end{equation*}
	is absolute winning, and thus the cocountable subset $S \setminus \mathcal{C}_F$ is also absolute winning.

	Let $x\in S\setminus\mathcal{C}_F$. Since $x\in\EFq$, the $F$-orbit of $x$ avoids some neighborhood of $q$ in $Y$. Thus choose $K\ge 2$ so large that
	\begin{equation*}
		F^n(x)\notin U_K
		\qquad\text{for every }n\ge0.
	\end{equation*}
	Because $x\notin\mathcal{C}_F$, no iterate of $x$ lies on a branch endpoint. Along the whole orbit of $x$, the maps $F$ and $F_K$ therefore agree. Then $x\in\mathcal E_{F_K}(\xi)$, so the $F_K$-orbit, and hence the $F$-orbit of $x$, avoids a relative neighborhood of $\xi$. Therefore $x\in\mathcal E_F(\xi)$ which proves that
	\begin{equation*}
		S \setminus \mathcal{C}_F \subset \mathcal E_F(\xi).
	\end{equation*}
	Thus $\mathcal E_F(\xi)$ is absolute winning.
\end{proof}

Theorem \ref{thm:secondary} now follows immediately from Theorem \ref{thm:all-single-targets-F}.

\begin{proof}[Proof of Theorem \ref{thm:secondary}]
	By Theorem \ref{thm:all-single-targets-F}, each $\mathcal E_F(\xi)$ with $\xi\in Q$ is absolute winning. Therefore
	\begin{equation*}
		\mathcal E_F(Q)=\bigcap_{\xi\in Q}\mathcal E_F(\xi)
	\end{equation*}
	is absolute winning.
\end{proof}

\section{Pulling the result back to the original map}\label{sec:pullback}

We now transfer the induced results to $f$. The transfer uses the invariance of the absolute winning property under quasisymmetric homeomorphisms.

Recall the definition of the decreasing sequence $r_n \searrow 0$ in $[0,1]$:
\begin{equation*}
	r_0 \coloneqq 1, \qquad f(r_{n+1}) = r_n.
\end{equation*}
For each $N\ge 0$, let
\[
h_N:[r_{N+1},r_N]\to[r_1,1]
\]
be the continuous extension of
\[
f^N\restriction_{(r_{N+1},r_N)}:(r_{N+1},r_N)\to(r_1,1),
\]
where $f^0$ is the identity on $[0,1]$.

\begin{lem}\label{lem:branch-quasisym}
	Each map $h_N$ with $N\ge 0$ is quasisymmetric, with a control function depending on $N$.
\end{lem}

\begin{proof}
	The case $N=0$ is trivial. For $N\ge1$, each $h_N$ is a $C^1$ diffeomorphism on $(r_{N+1},r_N)$ and Young's distortion estimate \cite[Section 6.2]{Young1999} gives a bound, depending on $N$, for
	\begin{equation*}
		\left\lvert\log\frac{\lvert(f^N)'(x)\rvert}{\lvert(f^N)'(y)\rvert}\right\rvert,
		\qquad x,y\in(r_{N+1},r_N).
	\end{equation*}
	A one-dimensional $C^1$ diffeomorphism with bounded derivative distortion on an interval is quasisymmetric.
\end{proof}

The next two lemmas show that exceptional sets under $F$ pull back into exceptional sets under $f$. Let
\begin{equation*}
	\mathcal{C}_f \coloneqq \bigcup_{n=0}^\infty f^{-n}(\{0,1\})
\end{equation*}
denote the countable collection of preimages under $f$ of the endpoint set $\{0,1\}$, equivalently of all branch endpoints. Since \(f(r_1)=0\), \(\mathcal C_f\) contains the sequence $\{r_n\}_{n=0}^\infty$.

\begin{lem}\label{lem:pullback-zero}
	For every $N\ge 0$,
	\begin{equation*}
		h_N^{-1}(\EFq\setminus\mathcal{C}_F) \subset \Ef \cup \mathcal{C}_f.
	\end{equation*}
\end{lem}

\begin{proof}
	Let $x\in h_N^{-1}(\EFq\setminus\mathcal{C}_F)$ and set $x'=h_N(x)$ so that $x'\in\EFq\setminus\mathcal{C}_F$. We may assume $x \notin \mathcal{C}_f$ and thus $x \in (r_{N+1},r_N)$. Then there is $K\ge 1$ such that the $F$-orbit of $x'$ avoids the one-sided neighborhood $U_K^\circ$ of $q$. Since $x'\notin\mathcal{C}_F$, the induced itinerary of $x'$ never encounters a cylinder endpoint, so the successive visits of the $f$-orbit of $x$ to the inducing base after time $N$ are exactly the points $F^j(x')$ for $j\geq1$.

	Choose $M>\max\{K,N\}$. We claim that the $f$-orbit of $x$ avoids $[0,r_M)$. If not, let
	\begin{equation*}
		t=\min\{n\ge 0:f^n(x)\in[0,r_M)\}.
	\end{equation*}
	Since $x\in(r_{N+1},r_N)$ and $M>N$, we have $t>N$. Write
	\[
		t=N+m_1+\cdots+m_j+s,
		\qquad
		0\le s<m_{j+1},
	\]
	where \(m_i=\tau(F^{i-1}(x'))\) are the successive return times in the \(F\)-itinerary of \(x'\). Then
	\[
		f^t(x)=f^s(F^j(x')).
	\]
	Since the \(F\)-orbit of \(x'\) avoids \(U_K^\circ\) and \(x'\notin\mathcal{C}_F\), each return time satisfies \(m_i<K<M\). In particular, the point \(f^t(x)\) returns to the inducing base after
	\[
		m_{j+1}-s < M
	\]
	further iterates. But \(f^t(x)\in[0,r_M)\) and $x \notin \mathcal{C}_f$; hence the first entrance of $f^t(x)$ into \([r_1,1]\) cannot occur in fewer than \(M\) iterates. This contradiction proves that the \(f\)-orbit of \(x\) avoids \([0,r_M)\) and thus $x \in \mathcal{E}_f(0)$.
\end{proof}

For $p \in [0,1]$ let
\begin{equation*}
	\mathcal{C}_f^p \coloneqq \bigcup_{j=0}^\infty f^{-j}(p)
\end{equation*}
denote the countable collection of preimages of $p$ under $f$.

\begin{lem}\label{lem:pullback-target}
	For every $N\ge0$ and every $p \in [0,1] \setminus \mathcal{C}_f$,
	\begin{equation*}
		h_N^{-1}\bigl(\mathcal E_F\big(f^{\tau(p)}(p)\big)\setminus\mathcal C_F\bigr)
		\subset \mathcal E_f(p) \cup \mathcal{C}_f^p.
	\end{equation*}
\end{lem}

\begin{proof}
	Let $x \in h_N^{-1}\bigl(\mathcal E_F\big(f^{\tau(p)}(p)\big)\setminus\mathcal C_F\bigr)$ and set $x'=h_N(x)$ so that $x' \in \mathcal E_F(f^{\tau(p)}(p))\setminus\mathcal C_F$. We may assume $x \notin \mathcal{C}_f^p$. Then the $F$-orbit of $x'$ avoids a neighborhood $V$ of $f^{\tau(p)}(p)$ in $[r_1,1]$.

	Since $p\notin\mathcal C_f$, the branch of $f^{\tau(p)}$ containing $p$ is well defined and $f^{\tau(p)}$ is continuous at $p$. Find a neighborhood $U_0$ of $p$ in $[0,1]$ such that every point of $U_0$ has the same itinerary as $p$ up to time $\tau(p)$ and such that
	\begin{equation*}
		f^{\tau(p)}(U_0)\subset V.
	\end{equation*}
	Since $x\notin\mathcal C_f^p$, no forward iterate of $x$ is equal to $p$. In particular, the finite set
	\begin{equation*}
		\{x,f(x),\ldots,f^N(x)\}
	\end{equation*}
	does not contain $p$. Therefore choose a neighborhood $U\subset U_0$ of $p$ such that
	\begin{equation*}
		U\cap \{x,f(x),\ldots,f^N(x)\}=\varnothing.
	\end{equation*}

	Suppose, toward a contradiction, that the $f$-orbit of $x$ enters $U$. Then there is some $t\geq0$ such that
	\begin{equation*}
		f^t(x)\in U.
	\end{equation*}
	By the choice of $U$, we must have $t>N$. Since $U\subset U_0$, it follows that
	\begin{equation*}
		f^{t+\tau(p)}(x)\in V\subset [r_1,1].
	\end{equation*}
	This is a visit of the $f$-orbit of $x$ to the inducing base after time $N$. Since $x'=f^N(x)\notin\mathcal C_F$, all successive visits of the $f$-orbit of $x$ to the inducing base after time $N$ are exactly the points $F^j(x')$ for $j\geq1$. But this contradicts the fact that $F^j(x')\notin V$ for every $j\geq0$. Therefore the $f$-orbit of $x$ avoids the neighborhood $U$ of $p$ and so $x\in\mathcal E_f(p)$.
\end{proof}

We now prove absolute winning for exceptional sets targeting points in $[0,1]$. The next few results handle targets in the preimage set $\mathcal{C}_f$.

\begin{thm}\label{thm:zero-target-f}
	The set $\Ef$ is absolute winning in $[0,1]$.
\end{thm}

\begin{proof}
	Let
	\begin{equation*}
		\mathcal{R}:=\{0\}\cup\{r_n\}_{n=0}^\infty
	\end{equation*}
	The components of $[0,1]\setminus \mathcal{R}$ are the intervals $(r_{N+1},r_N)$. For each $N\ge0$, Theorem \ref{thm:Efqwinning}, together with Lemma \ref{lem:branch-quasisym}, imply that
	\begin{equation*}
		h_N^{-1}(\EFq\setminus\mathcal{C}_F)
	\end{equation*}
	is absolute winning in $[r_{N+1},r_N]$. By Lemma \ref{lem:pullback-zero}, this set is contained in $S \coloneqq\Ef \cup \mathcal{C}_f$. Hence, $S$ is absolute winning locally in every component of $[0,1]\setminus \mathcal{R}$. Thus $S$ is absolute winning by Lemma \ref{lem:locality} and the result follows.
\end{proof}

\begin{lem}\label{lem:endpoint-one}
	The set $\mathcal E_f(1)$ is absolute winning in $[0,1]$.
\end{lem}

\begin{proof}
	By Theorem \ref{thm:all-single-targets-F}, the set $\mathcal E_F(1)\setminus\mathcal C_F$ is absolute winning in $[r_1,1]$. Then, for each $N\ge0$, the set
	\[
		h_N^{-1}\bigl(\mathcal E_F(1)\setminus\mathcal C_F\bigr)
	\]
	is absolute winning in $[r_{N+1},r_N]$ by Lemma \ref{lem:branch-quasisym}.

	We claim that
	\begin{equation*}
		h_N^{-1}\bigl(\mathcal E_F(1)\setminus\mathcal C_F\bigr) \subset \mathcal E_f(1)\cup\mathcal C_f^1.
	\end{equation*}
	Indeed, let $x$ lie in the displayed pullback and suppose $x\notin\mathcal C_f^1$. Put $x'=h_N(x)$. Since $x'\in\mathcal E_F(1)\setminus\mathcal C_F$, the $F$-orbit of $x'$ avoids some neighborhood $V_0$ of $1$ in $[r_1,1]$. Since $x \notin \mathcal{C}_f^1$ we can choose a smaller neighborhood $V \subset V_0$ of $1$ in $[0,1]$ such that
	\[
		V\cap\{x,f(x),\ldots,f^N(x)\}=\emptyset.
	\]
	If the $f$-orbit of $x$ entered $V$ at some time $t \geq 0$, then $t> N$; and since $V\subset[r_1,1]$, this would be a visit to the inducing base after time $N$. Because $x'\notin\mathcal C_F$, such visits are exactly the points of the induced orbit of $x'$. This contradicts the fact that the $F$-orbit of $x'$ avoids $V$. Hence the $f$-orbit of $x$ avoids $V$, so $x\in\mathcal E_f(1)$, proving the claim.

	Thus, $\mathcal E_f(1)\cup\mathcal C_f^1$ is absolute winning locally on each component interval $(r_{N+1},r_N)$. By Lemma~\ref{lem:locality}, applied with
	\[
		C=\{0\}\cup\{r_n\}_{n=0}^\infty,
	\]
	the set \(\mathcal E_f(1)\cup\mathcal C_f^1\) is absolute winning in \([0,1]\) and the result follows.
\end{proof}

\begin{lem}\label{lem:branch-point}
	The set $\mathcal E_f(r_1)$ is absolute winning in $[0,1]$.
\end{lem}

\begin{proof}
	We claim that
	\[
		\bigl(\mathcal E_f(0)\cap\mathcal E_f(1)\bigr)\setminus \mathcal C_f
		\subset \mathcal E_f(r_1).
	\]
	Indeed, suppose $x\notin\mathcal E_f(r_1)$ and $x\notin\mathcal C_f$. Then there are integers $n_j\to\infty$ such that
	\[
		f^{n_j}(x)\to r_1.
	\]
	Since $r_1\in\mathcal C_f$, the condition $x\notin\mathcal C_f$ implies that no iterate of $x$ is equal to $r_1$. Passing to a subsequence, we may therefore assume that the points $f^{n_j}(x)$ approach $r_1$ from one fixed side. Along that one-sided branch, the corresponding one-sided limit of $f$ at $r_1$ is an endpoint of $[0,1]$, namely either $0$ or $1$. Therefore
	\[
		f^{n_j+1}(x)\to 0
		\qquad\text{or}\qquad
		f^{n_j+1}(x)\to 1.
	\]
	Thus $x\notin\mathcal E_f(0)$ or $x\notin\mathcal E_f(1)$, proving the claim. The result now follows from Theorem \ref{thm:zero-target-f} and Lemma \ref{lem:endpoint-one}.
\end{proof}

\begin{thm}\label{thm:preimagesofendpoints}
	For every $p\in \mathcal{C}_f$, the set $\mathcal{E}_f(p)$ is absolute winning in $[0,1]$.
\end{thm}

\begin{proof}
	Let
	\[
		\mathcal A\coloneqq \{0,r_1,1\}.
	\]
	Fix $p\in\mathcal C_f$, and let $N\ge0$ be minimal such that
	\[
		f^N(p)\in\mathcal A.
	\]
	Write
	\[
		a\coloneqq f^N(p).
	\]
	We claim that
	\[
		\mathcal E_f(0)\cap \mathcal E_f(r_1)\cap\mathcal E_f(1)\subset \mathcal E_f(p).
	\]
	If $N=0$, then $p=a\in\mathcal A$, and the claim is immediate. Suppose $N\ge1$. By the minimality of $N$, the finite orbit segment
	\[
		p, f(p), \ldots, f^{N-1}(p)
	\]
	does not meet $\mathcal A$. In particular, it does not meet the branch point $r_1$. Hence there is a neighborhood $U$ of $p$ such that the iterates $f^k$, $0\le k\le N$, follow a single sequence of monotonicity branches on $U$. Therefore $f^N\restriction_U$ is continuous at $p$.

	Now suppose $x\notin\mathcal E_f(p)$. If $f^n(x)=p$ for some $n\ge0$, then
	\[
	f^{n+N}(x)=f^N(p)=a,
	\]
	so $x\notin\mathcal E_f(a)$, and hence
	\[
	x\notin \mathcal E_f(0)\cap \mathcal E_f(r_1)\cap\mathcal E_f(1).
	\]
	Otherwise, $p$ is not itself an iterate of $x$, so there are integers
	$n_j\to\infty$ such that
	\[
	f^{n_j}(x)\to p.
	\]
	For all sufficiently large $j$, we have $f^{n_j}(x)\in U$, and by the continuity of $f^N\restriction_U$ at $p$,
	\[
		f^{n_j+N}(x)=f^N(f^{n_j}(x))\to f^N(p)=a.
	\]
	Thus $x\notin\mathcal E_f(a)$. Since $a\in\{0,r_1,1\}$, it follows that
	\[
		x\notin \mathcal E_f(0)\cap \mathcal E_f(r_1)\cap\mathcal E_f(1).
	\]
	This proves the claim.

	By Theorem \ref{thm:zero-target-f} and Lemmas \ref{lem:endpoint-one} and \ref{lem:branch-point}, the set
	\[
		\mathcal E_f(0)\cap \mathcal E_f(r_1)\cap\mathcal E_f(1)
	\]
	is absolute winning. The result now follows from the claimed inclusion.
\end{proof}

\begin{thm}\label{thm:nonpreimage-targets-f}
	For every $p\in [0,1] \setminus \mathcal{C}_f$, the set $\mathcal E_f(p)$ is absolute winning in $[0,1]$.
\end{thm}

\begin{proof}
	Fix $p\in[0,1]\setminus\mathcal C_f$. By Theorem \ref{thm:secondary}, the set
	\begin{equation*}
		\mathcal E_F(f^{\tau(p)}(p))\setminus\mathcal C_F
	\end{equation*}
	is also absolute winning in $[r_1,1]$. For each $N\ge0$, Lemma \ref{lem:branch-quasisym} implies that
	\[
		h_N^{-1}\bigl(\mathcal E_F(f^{\tau(p)}(p))\setminus\mathcal C_F\bigr)
	\]
	is absolute winning in $[r_{N+1},r_N]$. Lemma \ref{lem:pullback-target} then gives that
	\begin{equation*}
		S \coloneqq \mathcal{E}_f(p) \cup \mathcal{C}_f^p
	\end{equation*}
	is absolute winning locally on each component interval $(r_{N+1},r_N)$. Applying Lemma \ref{lem:locality} with
	\[
		C=\{0\}\cup\{r_n\}_{n=0}^\infty
	\]
	gives that $S$ is absolute winning in $[0,1]$ and the result follows.
\end{proof}

\begin{thm}\label{thm:all-single-targets-f}
	For every $p\in[0,1]$, the set $\mathcal E_f(p)$ is absolute winning in $[0,1]$.
\end{thm}

\begin{proof}
	This follows immediately from Theorems \ref{thm:preimagesofendpoints} and \ref{thm:nonpreimage-targets-f}.
\end{proof}

Our main result, that $\mathcal{E}_f(P)$ is absolute winning whenever $P \subset [0,1]$ is countable, now follows easily from the preceding.

\begin{proof}[Proof of Theorem \ref{thm:main}]
	By Theorem \ref{thm:all-single-targets-f}, each $\mathcal E_f(p)$ with $p\in P$ is absolute winning. Since absolute winning is stable under countable intersections,
	\begin{equation*}
		\mathcal E_f(P)=\bigcap_{p\in P}\mathcal E_f(p)
	\end{equation*}
	is absolute winning.
\end{proof}

\end{document}